\documentclass[a4paper,11pt]{article}
\usepackage[utf8]{inputenc}

\usepackage{boxedminipage}
\usepackage{amsfonts}
\usepackage{amsmath} 
\usepackage{amssymb}
\usepackage{graphicx}
\usepackage{amsthm}
\usepackage{t1enc}
\usepackage{subfig}
\usepackage{authblk} 
 \usepackage{mathabx}

\newtheorem{theorem}{Theorem}[section]
\newtheorem{lemma}[theorem]{Lemma}

\newtheorem{definition}[theorem]{Definition}

\newtheorem{fact}[theorem]{Fact}
\newtheorem{question}{Question}

\newcommand{\forceP}{\mathbb{P}}
\newcommand{\forceQ}{\mathbb{Q}}
\newcommand{\forceR}{\mathbb{R}}

\newcommand{\ZFP}{\mathsf{ZF}^-}

\newcommand{\BPFA}{\mathsf{BPFA}}

\newcommand{\CH}{\mathsf{CH}}

\def\undertilde#1{\mathord{\vtop{\ialign{##\crcr
$\hfil\displaystyle{#1}\hfil$\crcr\noalign{\kern1.5pt\nointerlineskip}
$\hfil\tilde{}\hfil$\crcr\noalign{\kern1.5pt}}}}}

\title{$\mathsf{MA} (\mathcal{I}$) and a Failure of Separation on the third Level.} 
\author{ Stefan Hoffelner$^{1}$\footnote{This research was funded in whole by the Austrian Science Fund (FWF) Grant-DOI 10.55776/P37228. }  }
\date{
    $^1$TU Wien \\
    \today
}

\begin{document}

\maketitle
\begin{abstract}
We present a method which forces the failure of $\Pi^1_3$ and $\Sigma^1_3$-separation, while $\mathsf{MA} (\mathcal{I}$) holds, for $\mathcal{I}$ the family of indestructible ccc forcings. This shows that, in contrast to the assumption $\BPFA$ and $\aleph_1=\aleph_1^L$ which implies $\Pi^1_3$-separation, that weaker forcing axioms do not decide separation on the third projective level.
\end{abstract}

\section{Introducion}
Separation is one of the three classical and fundamental properties of definable subsets of the reals with numerous applications in descriptive set theory, measure theory, functional analysis and ergodic theory (see  e.g. the classical textbooks \cite{Moschovakis} or \cite{Kechris}).

Recall first the definition of separation.
\begin{definition}
 We say that a projective pointclass $\Gamma \in \{ \Sigma^1_n \mid n \in \omega\} \cup \{\Pi^1_n \mid n \in \omega \}$ has the separation 
 property (or just separation) iff every pair $A_0$ and $A_1$ of disjoint elements of  $\Gamma$ has a separating set $C$, i.e. a set $C$ such that $A_0 \subset C$ and $A_1 \subset \omega^{\omega} \setminus C$ and such that $C
   \in \Gamma \cap \check{\Gamma}$, where  $\check{\Gamma}$ denotes the dual pointclass of $\Gamma$.
\end{definition}
Now we turn to uniformization. Recall that for an $A \subset 2^{\omega} \times 2^{\omega}$, we say that $f$ is a uniformization (or a uniformizing function) of
$A$ if $f: pr_1 (A) \rightarrow 2^{\omega}$, where $pr_1(A)$ is $A$'s projection on the first coordinate, and the graph of $f$ is a subset of $A$. In other words, $f$ chooses exactly one point of every non-empty section of $A$.

\begin{definition}
 We say that the $\Gamma$-uniformization property is true, if 
every set $A \subset 2^{\omega} \times 2^{\omega}$, $A \in \Gamma$ has a uniformizing function $f_A$ whose graph is $\Gamma$.
\end{definition}
It is a classical result of Novikov that if $\Gamma$ has the uniformization property, then $\Gamma$ does not have the separation property.

Starting point for this work is the following, on the surface surprising theorem (see \cite{Ho1}):
\begin{theorem}
Assume $\BPFA$ and $\aleph_1=\aleph_1^{L}$. Then the  ${\Sigma}^1_3$-uniformization property holds, hence $\Pi^1_3$-separation holds as well and $\Sigma^1_3$-separation fails.
\end{theorem}
It is natural to ask whether already weaker forcing axioms axioms are sufficient to settle separation on the third projective level in the same way. We show that they do not suffice. Let 
$\mathsf{MA} (\mathcal{I} )$ be Martin's axiom for the class of indestructible ccc forcings, i.e. for those forcing $\forceP$ such that $\forceP \times \forceQ$ is ccc for every ccc forcing $\forceQ$. The main theorem of this article is the following:
\begin{theorem}
There is a model of $\mathsf{MA} (\mathcal{I} )$ where  $\aleph_1=\aleph_1^L$ and both $\Sigma^1_3$ and $\Pi^1_3$-separation fail.
\end{theorem}

Forcing a failure of $\Pi$-separation in the presence of forcing axioms is considerably more delicate than it might appear on a first glance. Aside from theorem 1.4, this is additionally explained by the fact that there is a very general method to force the global $\Sigma$-uniformization property (see \cite{Ho6}) and $\Sigma$-uniformization implies $\Pi$-separation. Thus any attempt to force the failure of $\Pi$-separation together with some forcing axiom, must be such that it can not be combined with the method from \cite{Ho6}, which is albeit technicalities a straightforward iteration with lots of flexibility. So failures of $\Pi$-separation in the presence of forcing axioms are necessarily in need of somewhat unconventional constructions. This article is an attempt to further explore the tension between forcing axioms and separation in the presence of the anti-large cardinal axiom $\omega_1=\omega_1^L$.

There exists literature already which deals with the failure of separation. L. Harrington, in handwritten notes from 1974 devised a forcing which forces the 
 failure of ${\Pi^1_n}$ and ${\Sigma^1_n}$-separation for a fixed $n \ge 3$. This result was proved in detail in \cite{KL} using different methods. Both constructions are very specific universes of set theory and the proof of the failure of separation involves  homogeneity considerations of the forcings involved, which will get lost once we start to throw in additional forcings in order to work for $\mathsf{MA} (\mathcal{I} )$. Thus in order to approach a proof of our main theorem a new technique is in need. The method presented here approaches the problem from a very different angle and builds on some techniques and ideas from \cite{Ho1}, \cite{Ho2}, \cite{Ho3}, \cite{Ho7}, in that we construct an iteration which produces two pairs of set on the third projective level which diagonalize against all possible pairs of $\Sigma^1_3$ or $\Pi^1_3$ sets which are candidates for  separating them.

\section{Independent Suslin trees in $L$, almost disjoint coding}
The coding method of our choice utilizes Suslin trees, which can be generically destroyed in an independent way of each other. 
\begin{definition}
 Let $\vec{T} = (T_{\alpha} \, : \, \alpha < \kappa)$ be a sequence of Suslin trees. We say that the sequence is an 
 independent family of Suslin trees if for every finite set of pairwise distinct indices $e= \{e_0, e_1,...,e_n\} \subset \kappa$ the product $T_{e_0} \times T_{e_1} \times \cdot \cdot \cdot \times T_{e_n}$ 
 is a Suslin tree again.
\end{definition}
Note that an independent sequence of Suslin trees $\vec{T} = (T_{\alpha} \, : \, \alpha < \kappa)$ has the property that whenever we decide to generically add branches to some of its members, then all the other members of $\vec{T}$ remain Suslin in the resulting generic extension. Indeed, if $A \subset \kappa$ and we form 
$ \prod_{i \in A} T_i $
with finite support, then in the resulting generic extension $V[G]$, for every $\alpha \notin A$, $V[G] \models `` T_{\alpha}$ is a Suslin tree$"$.

One can easily force the existence of independent sequences of Suslin trees with products of Jech's or Tennenbaum's forcing, or with just products of ordinary Cohen forcing. On the other hand independent sequences of length $\omega_1$ already exist in $L$.

Whenever we force with a Suslin tree $(T,<_T)$, i.e. we force with its nodes to add an uncountable branch, we denote the forcing with $T$ again.

We briefly introduce the almost disjoint coding forcing due to R. Jensen and R. Solovay (see \cite{JS}). We will identify subsets of $\omega$ with their characteristic function and will use the word reals for elements of $2^{\omega}$ and subsets of $\omega$ respectively.
Let $D=\{d_{\alpha} \, \: \, \alpha < \aleph_1 \}$ be a family of almost disjoint subsets of $\omega$, i.e. a family such that if $r, s \in D$ then 
$r \cap s$ is finite. Let $X\subset  \omega$  be a set of ordinals. Then there 
is a ccc forcing, the almost disjoint coding $\mathbb{A}_D(X)$ which adds 
a new real $x$ which codes $X$ relative to the family $D$ in the following way
$$\alpha \in X \text{ if and only if } x \cap d_{\alpha} \text{ is finite.}$$
\begin{definition}\label{definitionadcoding}
 The almost disjoint coding $\mathbb{A}_D(X)$ relative to an almost disjoint family $D$ consists of
 conditions $(r, R) \in [\omega]^{<\omega} \times D^{<\omega}$ and
 $(s,S) < (r,R)$ holds if and only if
 \begin{enumerate}
  \item $r \subset s$ and $R \subset S$.
  \item If $\alpha \in X$ and $d_{\alpha} \in R$ then $r \cap d_{\alpha} = s \cap d_{\alpha}$.
 \end{enumerate}
\end{definition}
We shall briefly discuss the $L$-definable, $\aleph_1^L$-sized almost disjoint family of reals $D$  we will use throughout this article. The family $D$ is the canonical almost disjoint family one obtains when recursively adding the $<_L$-least real $x_{\beta}$ not yet chosen and replace it with $d_{\beta} \subset \omega$ where that $d_{\beta}$  is the real which codes the initial segments of $x_{\beta}$ using some recursive bijections between $\omega$ and $\omega^{<\omega}$.

\section{The ground model $W$}

We have to first create a suitable ground model $W$ over which the actual iteration will take place. The construction is inspired by \cite{Ho2}, where a very similar universe is used as well. $W$ will be a generic extension of $L$, satisfying $\CH$ and {having} the crucial property that in $W$ there is an $\omega_2$-sequence $\vec{S}$ of independent Suslin trees which is $\Sigma_1(\omega_1)$-definable over $H(\omega_2)^W$. The sequence $\vec{S}$ will enable a coding method which is independent of the surrounding universe, a feature we will exploit to a great extent in {what follows}.

To form $W$, we start with G\"odels constructible universe $L$ as our 
ground model. Recall that $L$ comes equipped with a $\Sigma_1$-definable, global well-order $<_L$ of its elements.
We first fix an appropriate sequence of stationary, co-stationary subsets of $\omega_1$ using Jensen's $\diamondsuit$-sequence. The proof of the next lemma is well-known so we skip it.
\begin{fact}
In $L$ there is a sequence $(a_{\alpha} \, : \, \alpha < \omega_1)$ of countable subsets of $\omega_1$
such that any set $A \subset \omega_1$ is guessed stationarily often by the $a_{\alpha}$'s, i.e.
$\{ \alpha < \omega_1 \, : \, a_{\alpha}= A \cap \alpha \}$ is a stationary and co-stationary subset of $\omega_1$. The sequence $(a_{\alpha} \, : \, \alpha < \omega_1)$ can be defined in a $\Sigma_1$ way over the structure $L_{\omega_1}$.
\end{fact}

The $\diamondsuit$-sequence can be used to produce an easily definable sequence of $L$-stationary, co-stationary subsets of $\omega_1$: we list the elements of $L_{\omega_2}$ in an $\omega_2$ sequence $(r_{\alpha} \, : \, \alpha < \omega_2)$. 
We fix \[R:= \{\alpha < \omega_1 \, : \, a_{\alpha}= r_0 \cap \alpha \} \] and note that $R$ is stationary and co-stationary. 

Then we define for every $\beta < \omega_2$, $\beta \ne 0$
a stationary, co-stationary set in the following way:
\[R'_{\beta} := \{ \alpha < \omega_1 \, : \, a_{\alpha}= {r}_{\beta} \cap \alpha \}\]
and 
\[R_{\beta} := \{ \alpha < \omega_1 \, : \, a_{\alpha}= {r}_{\beta} \cap \alpha \} \backslash R.\]
 It is clear that $\forall \alpha \ne \beta (R_{\alpha} \cap R_{\beta} \in \hbox{NS}_{\omega_1})$ and that $R_{\beta} \cap R=\emptyset$. 
To avoid writing $\beta \ne 0$ all the time we re-index and confuse $(R_{\beta} \, : \, \beta < \omega_2, \beta \ne 0)$ with $(R_{\beta} \, : \, \beta < \omega_2)$.
We derive the following standard result concerning the definability of the $R_{\beta}$'s. A detailed proof can be found in \cite{Ho1}, Lemma 1.12.
\begin{lemma}\label{computationofRbetas}
For any $\beta < \omega_2$, membership in $R_{\beta}$ is uniformly $\Sigma_1(\omega_1)$-definable over the model $L_{\omega_2}$, i.e. there is a $\Sigma_1$-formula with $\omega_1$ as a parameter $\psi(v_0,v_1, \omega_1)$ such that for every $\beta < \omega_2$,
$(\alpha \in R_{\beta} \Leftrightarrow L_{\omega_2} \models \psi(\alpha, \beta,\omega_1))$. 
\end{lemma}

We proceed with defining the universe $W$.
First we generically add $\aleph_2$-many Suslin trees using of Jech's Forcing $ \forceP_J$. We let 
\[\forceQ^0 := \prod_{\beta \in \omega_2} \forceP_J \] using countable support. This is a $\sigma$-closed, hence proper notion of forcing. We denote the generic filter of $\forceQ^0$ with $\vec{S}=(S_{\alpha} \, : \, \alpha < \omega_2)$ and note that $\vec{S}$ is independent.  We fix a definable bijection between {$[\omega_1]^{2}$ and $\omega_1$} and identify the trees in {$(S_{\alpha }\, : \, \alpha < \omega_2)$} with their images under this bijection, so the trees will always be subsets of $\omega_1$ from now on.

In a second step we code the trees from $\vec{S}$ into the sequence of $L$-stationary subsets $\vec{R}$ we produced earlier. It is important to note that the forcing we are about to define does preserve Suslin trees, a fact we will show later.
The forcing used in the second step will be denoted by $\mathbb{Q}^1$ and will be a countable support iteration of length {$\omega_2$} whose factors are itself countably supported iterations of length $\omega_1$. Fix first a definable bijection $h \in L_{\omega_3}$ between $\omega_1 \times \omega_2$ and $\omega_2$ and write $\vec{R}$ from now on in ordertype $\omega_1 \cdot \omega_2$ making implicit use of $h$, so we assume that $\vec{R}= (R_{\alpha} \, : \, \alpha < \omega_1 \cdot \omega_2)$. 

$\forceQ^1$ is defined via induction in the ground model $L[\forceQ^0]$. Assume we are at stage $\alpha < \omega_2$ and we have already created the iteration $\forceQ^1_{\alpha}$ up to stage $\alpha$. We work with $L[\forceQ^0][\forceQ^1_{\alpha}]$ as our ground model
 and consider the Suslin tree $S_{\alpha} \subset \omega_1$. 
 We define the forcing we want to use at stage $\alpha$, denoted by $\forceQ^1 (\alpha)$, as the countable support iteration which codes the characteristic function of $S_{\alpha}$ into the $\alpha$-th $\omega_1$-block of the $R_{\beta}$'s. So $\forceQ^1 (\alpha)= \Asterisk_{\gamma < \omega_1} \forceR_{\alpha}(\gamma)$ is again a countable support iteration in $L[\forceQ^0][\forceQ^1_{\alpha}]$, defined inductively via
\[ \forall \gamma < \omega_1 \,(\forceR_{\alpha}(\gamma):= \dot{\forceP}_{\omega_1 \backslash R_{\omega_1 \cdot \alpha + 2 \gamma +1}}) \text{ if } S_{\alpha} (\gamma) =0 \]
and
\[ \forall \gamma < \omega_1 \, (\forceR_{\alpha}(\gamma):= \dot{\forceP}_{\omega_1 \backslash R_{\omega_1 \cdot \alpha + 2 \gamma}}) \text{ if } S_{\alpha} (\gamma) =1. \]

Recall that we let $R$ be a stationary, co-stationary subset of $\omega_1$ which is disjoint from all the $R_{\alpha}$'s which are used.  We let $\mathbb{Q}^1$ be the countably supported iteration, $$\mathbb{Q}^1:=\Asterisk_{\alpha< \omega_2} \forceQ^1_{\alpha}$$ which is an $R$-proper forcing. We shall see later that $\forceQ^1$ in fact is $\omega$-distributive, hence the iteration $\forceQ^1$ is in fact a countably supported product.
This way we can turn the generically added sequence of Suslin trees $\vec{S}$ into a definable sequence of Suslin trees.
Indeed, if we work in $L[\vec{S}\ast G]$, where $\vec{S} \ast G$ is $\forceQ^0 \ast \mathbb{Q}^1$-generic over $L$, then
\begin{align*}
\forall \alpha< \omega_2, \gamma < \omega_1 (&\gamma \in S_{\alpha} \Leftrightarrow R_{\omega_1 \cdot \alpha + 2 \cdot \gamma} \text{ is not stationary and} \\ &
\gamma \notin S_{\alpha} \Leftrightarrow  R_{\omega_1 \cdot \alpha + 2 \cdot \gamma +1} \text{ is not stationary})
\end{align*}
Note here that the above formula can be used to make every $S_{\alpha}$ $\Sigma_1(\omega_1,\alpha)$ definable over $L[\vec{S} \ast G]$ as is shown with the next lemma.

\begin{lemma}\label{definabilityofvecS}
There is a $\Sigma_1(\omega_1)$-formula $\Phi(v_0,v_1,\omega_1)$ {such that for every} $\gamma < \omega_2$,

\[ L[\vec{S} \ast G] \models \forall \beta < \omega_1 (\Phi (\beta, \gamma,\omega_1) \Leftrightarrow \beta \in S_{\gamma}) \]
 
Thus initial segments  of the sequence $\vec{S} $ are uniformly  $\Sigma_1(\omega_1)$-definable over $L[\vec{S} \ast G]$.
\end{lemma}
\begin{proof}
Let $\gamma < \omega_2$.
We claim that already $\aleph_1$-sized, transitive models of $\ZFP$ which contain a club through the complement of exactly one element of every pair $\{(R_{\alpha}, R_{\alpha+1}) \, : \, \alpha < \omega_1 \cdot \gamma\}$ are sufficient to compute correctly $\vec{S} \upharpoonright \gamma$ via the following $\Sigma_1(\gamma, \omega_1)$-formula: 

\begin{align*}
\Psi(X,\gamma, \omega_1) \equiv  \exists M (&M \text{ transitive } \land M \models \ZFP \land \omega_1,\gamma \in M \land
\\& M \models \forall \beta< \omega_1 \cdot \gamma (\text{either  $R_{2\beta}$ or $R_{2\beta+1}$ is nonstationary) } \land \\& 
M \models X \text{ is an $\omega_1 \cdot \gamma$-sequence $(X_{\alpha})_{\alpha < \omega_1 \cdot \gamma}$ of subsets of $\omega_1$} \land\\&
M \models  \forall \alpha, \delta (\delta \in X_{\alpha} \Leftrightarrow R_{\omega_1 \cdot \alpha + 2 \cdot \delta} \text{ is not stationary and} \\& \qquad \qquad \quad \,
\delta \notin X_{\alpha} \Leftrightarrow  R_{\omega_1 \cdot \alpha + 2 \cdot \delta +1} \text{ is not stationary})
\end{align*}
We want to show that 
\begin{align*}
X=\vec{S} \upharpoonright \gamma \text{ if and only if } \Psi(X,\gamma,\omega_1) \text{ is true in }L[\vec{S} \ast G].
\end{align*}

For the backwards direction, we assume that $M$ is a model and $X \in M$ is a set, as on the right hand side of the above. We shall show that indeed $X=\vec{S}\upharpoonright \gamma$.  As $M$ is transitive and a model of $\ZFP$ it will compute every $R_{\beta}$, $\beta < \omega_1 \cdot \gamma$ correctly by Lemma \ref{computationofRbetas}. As being nonstationary is a $\Sigma_1(\omega_1)$-statement, and hence upwards absolute, we conclude that if $M$ believes to see a pattern written into (its versions of) the $R_{\beta}$'s, this pattern is exactly the same as is seen by the real world $L[\vec{S} \ast G]$. But we know already that in $L[\vec{S} \ast G]$, the sequence $\vec{S}$ is written into the $R_{\beta}$'s, thus $X=\vec{S} \upharpoonright \gamma$ follows.

On the other hand, if $X=\vec{S} \upharpoonright \gamma$, then
\begin{align*}
L[\vec{S} \ast G] \models \forall \beta< \omega_1 \cdot \gamma (\text{either  $R_{2\beta}$ or $R_{2\beta+1}$ is nonstationary) } \\ 
L[\vec{S} \ast G]  \models X \text{ is an $\omega_1 \cdot \gamma$-sequence $(X_{\alpha})_{\alpha < \omega_1 \cdot \gamma}$ of subsets of $\omega_1$}
\end{align*}
and
\begin{align*}
L[\vec{S} \ast G] \models  \forall \alpha, \delta < \omega_1 (\delta  \in X_{\alpha} \Leftrightarrow & R_{\omega_1 \cdot \alpha + 2 \cdot \delta } \text{ is not stationary and} \\ 
\delta  \notin X_{\alpha} \Leftrightarrow&  R_{\omega_1 \cdot \alpha + 2 \cdot \delta +1} \text{ is not stationary})
\end{align*}
By reflection, there is an $\aleph_1$-sized, transitive model $M$ which models the assertions above, which gives the direction from left to right.

\end{proof}

Let us set \[W:= L[\forceQ^0\ast \forceQ^1 ]\] which will serve as our ground model for a second iteration of length $\omega_2$. We shall need the following
well-known result (see \cite{Abraham}, Theorem 2.10 {pp. 20}.)
\begin{fact}
Let $R\subset \omega_1$ be stationary and co-stationary.
Assume $\CH$ and let $(\forceP_{\alpha} \, : \, \alpha \le \delta)$ be a
countable support iteration of $R$-proper posets such
that for every $\alpha \le \delta$
\[ \forceP_{\alpha} \Vdash |\forceP(\alpha)| = \aleph_1. \]
Then $\forceP_{\delta}$ satisfies the $\aleph_2$-cc.
\end{fact}

\begin{lemma}
$W$ is an $\omega$-distributive generic extension of $L$ which also satisfies $\aleph_2^L=\aleph_2^W$.
\end{lemma}
\begin{proof}
The second assertion follows immediately from the last {fact}. The first assertion holds by the following argument which already will look familiar. First note that as $\forceQ^0$ does not add any reals it is sufficient to show that $\forceQ^1$ is $\omega$-distributive.
Let $p \in \forceQ^1$ be a condition and assume that $p \Vdash \dot{r} \in 2^{\omega}$. We shall find a stronger $q < p$ and a real $r$ in the ground model such that $q \Vdash \check{r}=\dot{r}$. Let $M \prec H(\omega_3)$ be a countable elementary submodel which contains $p, \forceQ^1$ and $\dot{r}$ and such that $M \cap \omega_1 \in R$, where $R$ is our fixed stationary set from above. Inside $M$ we recursively construct a decreasing sequence $p_n$ of conditions in $\forceQ^1$, such that for every $n$ in $\omega,$ $p_n \in M$, $p_n$ decides $\dot{r}(n)$ and for every $\alpha$ in the support of $p_n$, the sequence $sup_{n \in \omega} max( p_n(\alpha))$ converges towards $M \cap \omega_1$ which is in $R$. Now, $q':= \bigcup_{n \in \omega} p_n$ and for every $\alpha< \omega_1$ such that $q'(\alpha)\ne 1$ (where 1 is the weakest condition of the forcing),  in other words for every $\alpha$ in the support of $q'$ we define $q(\alpha):= q'(\alpha) \cup \{((M \cap \omega_1), (M \cap \omega_1))\}$ and $q(\alpha)=1$ otherwise. Then $q=(q(\alpha))_{\alpha < \omega_1}$ is a condition in $\forceQ^1$, as can be readily verified and $q \Vdash \dot{r} = \check{r}$, as desired.
\end{proof}
Note that by the last lemma, the second forcing $\forceQ^1$ which is {a} countably supported iteration of the appropriate club shooting forcings is in fact just a countably supported product of its factors, i.e. at every stage of the iteration we can force with the associated $\forceP_R$, as computed in $L$.

Our goal is to use $\vec{S}$ for coding again. For this it is essential that the sequence remains independent in $W$. To see this we shall argue that forcing with $\mathbb{Q}^1$ over $L[\forceQ^0]$ preserves Suslin trees. 
Recall that for a forcing $\forceP$, $\theta$ sufficiently large and regular and $M \prec H(\theta)$, a condition $q \in \forceP$ is $(M,\forceP)$-generic iff for every maximal antichain $A \subset \forceP$, $A \in M$, it is true that $ A \cap M$ is predense below $q$. In the following we will write $T_{\eta}$ to denote the $\eta$-th level of the tree $T$ and $T \upharpoonright \eta$ to denote the set of nodes of $T$ of height $< \eta$.
The key fact is the following (see \cite{Miyamoto2} for the case where $\forceP$ is proper, see also \cite{Ho1} for a detailed proof){.} 
\begin{lemma}\label{preservation of Suslin trees}
 Let $T$ be a Suslin tree, $R \subset \omega_1$ stationary and $\forceP$ an $R$-proper
 poset. Let $\theta$ be a sufficiently large cardinal.
 Then the following are equivalent:
 \begin{enumerate}
  \item $\Vdash_{\forceP} T$ is Suslin
 
  \item if $M \prec H_{\theta}$ is countable, $\eta = M \cap \omega_1 \in R$, and $\forceP$ and $T$ are in $M$,
  {and also} $p \in \forceP \cap M$, then there is a condition $q<p$ such that 
  for every condition $t \in T_{\eta}$, 
  $(q,t)$ is $(M, \forceP \times T)$-generic.
 \end{enumerate}

\end{lemma}

{With a routine adjustment one can show that Theorem 1.3 of} \cite{Miyamoto2}{ holds true if one replaces proper by $R$-proper for $R \subset \omega_1$ a stationary subset, i.e. that a countable support iteration of $R$-proper forcings which preserve Suslin trees results in a forcing which preserves Suslin trees.
\begin{lemma}\label{omegadistributive}
Let $R \subset \omega_1$ be stationary, co-stationary, then the club shooting forcing $\forceP_R$ preserves Suslin trees.
\end{lemma}

Invoking the generalized version of Miyamoto's theorem we obtain:
\begin{lemma}
The second forcing $\forceQ^1$ we use to produce $W=L[\forceQ^0 \ast \forceQ^1]$ preserves Suslin trees. As a consequence the $\omega_1$-trees from $\vec{S}$ are Suslin trees in $W$.
\end{lemma}

We end with a straightforward lemma which is used later in coding arguments.

\begin{lemma}\label{a.d.coding preserves Suslin trees}
 Let $T$ be a Suslin tree and let $\mathbb{A}_D(X)$ be the almost disjoint coding which codes
 a subset $X$ of $\omega_1$ into a real with the help of an almost disjoint family
 of reals $D$ of size $\aleph_1$. Then $$\mathbb{A}_{D}(X) \Vdash_{} T \text{ is Suslin }$$
 holds.
\end{lemma}
\begin{proof}
 This is clear as $\mathbb{A}_{D}(X)$ has the Knaster property, thus the product $\mathbb{A}_{D}(X) \times T$ is ccc and $T$ must be Suslin in $V[{\mathbb{A}_{D}(X)}]$. 
\end{proof}

\section{Coding device}
The following is taken entirely from \cite{Ho7}, which itself is a simplification of the coding machinery first constructed in \cite{Ho2} and \cite{Ho4}. We continue with the construction of the appropriate notions of forcing which we want to use in our proof.

The objective is to first define a coding forcing \(\operatorname{Code}(x,\eta, i)\) for real numbers \( x \) and ordinals $\eta < \omega_2$ and $i \in 3$, which will force a certain \(\Sigma^1_3\)-formula \(\Phi(x,i)\) to hold in the resulting generic extension. The forcing will be a two step iteration $\forceQ_0 \ast \forceQ_1$ an will be defined uniformly for all four $\omega_2$-sequences of Suslin trees $\vec{S}^0,\vec{S}^1,\vec{S}^2$ and $\vec{S}^3$. In order to reduce notation we will define it just for $i=1$.

In the first step, we pick our given real $x \subset \omega$ and eliminate all elements of \(\vec{S}^1\) starting at the $\eta$-th $\omega$-block of \( \vec{S}^1 \), by generically adding an \(\omega_1\)-branch according  $x\subset \omega$. Specifically, we first form \(E_{x,\eta}:= \{ \omega \cdot \eta +2n \mid n \in \omega, n \in x \} \) and
\( F_{x,\eta} := \{ \omega \cdot \eta+2n+2 \mid n \in \omega, n \notin x \} \). Then we force with

 \[\forceQ_0:= \prod_{\alpha  \in E_{x,\eta} \cup F_{x,\eta}} S_{\alpha}\] with finite support with ground model \(W\). This is an \(\aleph_1\)-sized, ccc forcing over \(W\), so in the generic extension, \(\aleph_1\) is preserved. We let $g$ be a $\forceQ_0$-generic filter and work in $W[g]$.

Let \(X \subset \omega_1\) be the least set (in some fixed well-order of \(H(\omega_2)^W[g]\)) that encodes the following objects:
\begin{itemize}

\item  The $<$-least set of $\omega_1 \cdot \omega \cdot \omega_1$-many club subsets through $\vec{R}$, our $\Sigma_1 (\{\omega_1\})$-definable sequence of $L$-stationary subsets of $\omega_1$, which are necessary to compute every tree $S_{\beta} \in \vec{S}^1$ in the interval $[ \omega \cdot \eta, \omega\cdot \eta + \omega)$, using the $\Sigma_1 (\{\omega_1\})$-formula from the previous section.

\item The least set of \(\omega_1\)-branches in \(W\) through elements of \(\vec{S}^1\) that code \( x\) at the \(\omega\)-block of trees in $\vec{S}^1$ starting at  \( \eta \). Specifically, we collect:
\[
\{ b_{\beta} \subset S^1_{\beta} : \beta \in E_{x,\eta} \cup F_{x,\eta} \}
\]

\end{itemize}

In \(L[X]\),  we can decode \(x \) by examining the \(\omega\)-block of \(\vec{S}^1\)-trees starting at \(\eta\) and determining which tree has an \(\omega_1\)-branch in \(L[X]\):
\begin{itemize}
 \item[$(\ast)$]  $n \in x$ if and only if $S^1_{\omega \cdot \eta +2n+1}$ has an $\omega_1$-branch, and $n \notin x$ if and only if $S^1_{\omega \cdot \eta +2n}$ has an $\omega_1$-branch.
\end{itemize}
Indeed, if \(n \notin x\), we added a branch through \(S^1_{\omega \eta + 2n}\). If \(S^1_{\omega \eta + 2n}\) is Suslin in \(L[X]\), then we must have added an \(\omega_1\)-branch through \(S^1_{\omega \eta + 2n+1}\), since we always add an \(\omega_1\)-branch through either \(S^1_{\omega \eta + 2n+1}\) or \(S^1_{\omega \eta + 2n}\). Adding branches through some \(S^1_{\alpha}\)'s will not affect whether some \(S^1_{\beta}\) is Suslin in \(L[X]\), as \(\vec{S}^1\) is independent.

We can now apply an argument similar to David’s trick (see \cite{David} or \cite{Friedman}). We rewrite the information of \(X \subset \omega_1\) as a set \(Y \subset \omega_1\). Any transitive, \(\aleph_1\)-sized model \(M\) of \(\ZFP\) containing \(X\) will be able to decode all the information from \(X\). Therefore, if we code the model \((M, \in)\) containing \(X\) as a set \(X_M \subset \omega_1\), then for any uncountable \(\beta\) such that \(L_{\beta}[X_M] \models \ZFP\) and \(X_M \in L_{\beta}[X_M]\), we have:
\[
L_{\beta}[X_M] \models \text{``The model decoded from } X_M \text{ satisfies } (\ast)."
\]
In particular, there will be an \(\aleph_1\)-sized ordinal \(\beta\), and we can fix a club \(C \subset \omega_1\) and a sequence \((M_{\alpha} : \alpha \in C)\) of countable elementary submodels of \(L_{\beta}[X_M]\) such that:
\[
\forall \alpha \in C \, (M_{\alpha} \prec L_{\beta}[X_M] \land M_{\alpha} \cap \omega_1 = \alpha).
\]
We now define \(Y \subset \omega_1\) to code the pair \((C, X_M)\), where the odd entries of \(Y\) code \(X_M\), and \(E(Y)\) denotes the set of even entries of \(Y\). Let \(\{c_{\alpha} : \alpha < \omega_1\}\) be the enumeration of \(C\), and we define the following conditions:
\begin{itemize}
\item[-] \(E(Y) \cap \omega\) codes a well-ordering of type \(c_0\).
\item[-] \(E(Y) \cap [\omega, c_0) = \emptyset\).
\item[-] For each \(\beta\), \(E(Y) \cap [c_{\beta}, c_{\beta} + \omega)\) codes a well-ordering of type \(c_{\beta+1}\).
\item[-]For each \(\beta\), \(E(Y) \cap [c_{\beta} + \omega, c_{\beta+1}) = \emptyset\).
\end{itemize}

Finally, we obtain:
\begin{itemize}
\item[$({\ast}{\ast})$] For any countable transitive model $M$ of $\ZFP$ such that $\omega_1^M=(\omega_1^L)^M$ and $ Y \cap \omega_1^M \in M$, $M$ can construct its version of the universe $L[Y \cap \omega_1^N]$, and the latter will see that there is an $\aleph_1^M$-sized transitive model $N \in L[Y \cap \omega_1^N]$ which models $(\ast)$ for $x$ at some $\omega$-block of $\vec{S}^1$-trees.
\end{itemize}
This establishes a local version of the property \((\ast)\).

In the next step, we use almost disjoint forcing \(\mathbb{A}_D(Y)\) relative to to some canonically defined almost disjoint family of reals \(D \in L\) to code the set \(Y\) into one real \(r\). This forcing depends only on the subset of \(\omega_1\) being coded, so \(\mathbb{A}_D(Y)\) will be independent of the surrounding universe, as long as it has the correct \(\omega_1\) and contains \(Y\).

We finally obtain a real \(r\) such that 
\begin{itemize}
\item[$({\ast}{\ast}{\ast})$] For any countable, transitive model $M$ of $\ZFP$ such that $\omega_1^M=(\omega_1^L)^M$ and $ r  \in M$, $M$ can construct its version of $L[r]$ which in turn thinks that there is a transitive $\ZFP$-model $N$ of size $\aleph_1^M$  such that $N$ believes $(\ast)$ for $x$.
\end{itemize}
Note that $({\ast} {\ast} {\ast})$ is a $\Pi^1_2$-formula in the parameters $r$ and $x$. We will often suppress the reals $r,x$ when referring to $({\ast} {\ast} {\ast})$ as they will be clear from the context. We say in the above situation that the real $x$ \emph{ is written into $\vec{S}^1$} at some $\eta$, or that $x$ \emph{is coded into} $\vec{S^1}$ (at $\eta$) and $r$ witnesses that $x$ is coded. Likewise a forcing $\operatorname{Code} (x,i)$ is defined for coding the real $x$ into $\vec{S^i}$.

 The projective and local statement $({\ast} { \ast} {\ast} )$, if true,  will determine how certain inner models of the surrounding universe will look like with respect to branches through $\vec{S}^i$.
That is to say, if we assume that $({\ast} { \ast} {\ast} )$ holds for a real $x$ and is the truth of it is witnessed by a real $r$. Then $r$ also witnesses the truth of $({\ast} { \ast} {\ast} )$ for any transitive $\ZFP$-model $M$ which contains $r$ (i.e. we can drop the assumption on the countability of $M$).
Indeed if we assume 
that there would be an uncountable, transitive $M$, $r \in M$, which witnesses that $({\ast} { \ast} {\ast} )$ is false. Then by L\"owenheim-Skolem, there would be a countable $N\prec M$, $r\in N$ which we can transitively collapse to obtain the transitive $\bar{N}$. But $\bar{N}$ would witness that $({\ast} { \ast} {\ast} )$ is not true for every countable, transitive model, which is a contradiction.

Consequently, the real $r$ carries enough information that
the universe $L[r]$ will see that certain trees from $\vec{S}^i$ have branches in that
\begin{align*}
n \in x \Rightarrow L[r] \models  ``S^i_{\omega \eta + 2n+1} \text{ has an $\omega_1$-branch}".
\end{align*}
and
\begin{align*}
n \notin x\Rightarrow L[r] \models ``S^i_{\omega \eta + 2n} \text{ has an $\omega_1$-branch}".
\end{align*}
Indeed, the universe $L[r]$ will see that there is a transitive $\ZFP$-model $N$ which believes $(\ast)$, the latter being coded into $r$. But by upwards $\Sigma_1$-absoluteness, and the fact that $N$ can compute $\vec{S}^i$ correctly, if $N$ thinks that some tree in $\vec{S^1}$ has a branch, then $L[r]$ must think so as well.

\section{Proof of the theorem}

\subsection{Definition of the sets which eventually can not be separated}
We shall describe the ideas behind the construction. We use the independent Suslin trees $\vec{S}^0$ and $\vec{S}^1$ to define two disjoint $\Pi^1_3$-sets $D^0$, $D^1$ which can not be separated by a pair of disjoint $\Sigma^1_3$-sets $A_m, A_k$. The trees from $\vec{S}^2$ and $\vec{S}^3$ will be utilized to define two disjoint $\Sigma^1_3$-sets $E^2,E^3$ which can not be separated by a pair of disjoint $\Pi^1_3$-sets $B_m, B_k$. 
We let 
\[D^0:= \{ x \mid x \text{ is not coded into $\vec{S}^0$\}} \]
and
\[ D^1:= \{ x \mid x \text{ is not coded into $\vec{S}^1$ \}.} \]
Note that as being coded is a $\Sigma^1_3$ property, not being coded is $\Pi^1_3$ and so are $D^0$ and $D^1$.

Likewise we let
\[E^2:= \{ x \mid x \text{ is coded into $\vec{S}^2$ } \} \]
and
\[E^3:= \{ x  \mid x \text{ is coded into $\vec{S}^3$}\}. \]
which are two $\Sigma^1_3$-sets.

The forcings we intend to use belong to a certain specific set of forcings which we dub suitable and which shall be defined now:

\subsection{Suitable forcings}

\begin{definition}
Let $\delta< \omega_2$ and $F: \delta \rightarrow H(\omega_2)$ and $E: H(\omega_2) \rightarrow \{2,3\}$ be two functions. A finite support iteration of length $\delta$, $(\forceP_{\beta} ,\dot{\forceQ}_{\beta}\mid \beta \le \delta)$ is called suitable (with respect to $F$ and $E$ if it obeys the following rules. Suppose we arrived at stage $\beta$ of our iteration and we obtained already $\forceP_{\beta}$. Let $G_{\beta}$ be a $\forceP_{\beta}$-generic filter and assume that $F(\beta)= (\dot{x},m,k,i, \dot{\eta})$, where $\dot{x}$ is a $\forceP_{\beta}$-name of a real, $\dot{\eta}$ is a name of an ordinal and $m,k,i$ are natural numbers.
\begin{enumerate}
\item If $m,k$ are G\"odelnumbers of two $\Sigma^1_3$-formulas and $i \in 2$ then
we let $\dot{\forceQ}_{\beta}:= \operatorname{Code} (x,\eta,i)$, provided the $\eta$-th $\omega$-block of Suslin trees from $\vec{S}^1$ has not been used already for coding. Otherwise we force with $\operatorname{Code} (x,\eta',i)$, where $\omega \eta'$ is the least $\omega$-block of trees from $\vec{S}^i$ which has not been used for coding.
\item If $m,k$ are G\"odelnumbers of two $\Pi^1_3$-formulas and $i \in \{2,3\}$ then we let $x= \dot{x}^{G_{\beta}}$ and let $\tau$ be the $<$-least forcing name of $x$, where $<$ is some previously fixed wellorder. We calculate the value of $E(\tau) \in \{2,3\}$ and force with $\dot{\forceQ}_{\beta}:= \operatorname{Code} (x, \eta, E(\tau))$, provided the $\eta$-th $\omega$-block has not been used already for coding forcings. If not we use $\dot{\forceQ}_{\beta}:= \operatorname{Code} (x, \eta', E(\tau))$ where $\eta'$ is again the least ordinal such that its $\omega$-block of elements of $\vec{S}^{H(\tau)}$ has not been used for coding forcings.

\item If $F(\beta)$ is the $\forceP_{\beta}$-name of an $\aleph_1$-sized, ccc indestructible forcing $\dot{\forceQ}$, then we use $\dot{\forceQ}$ at stage $\beta$.
\end{enumerate}
\end{definition}

We note that two suitable forcings $\forceP^0$, $\forceP^1$ relative to a bookkeeping $F,F'$ and some common $E$ are not necessarily closed under taking products. Indeed it could be that both $\forceP^0$ and $\forceP^1$ use some $\omega$-block of trees form $\vec{S}^i$ but coding up different reals there. This, however is the only obstruction for closure under products, and we obtain that $\forceP^0 \times \forceP^1$ is suitable if and only if the set of trees used for coding by $\forceP^0$ and the set of trees used for coding by $\forceP^1$ are disjoint.

We will define an iteration of length $\omega_2$ as follows. First we will work towards a failure of $\Pi^1_3$-separation, more specifically we will create a model in which no pair of $\Sigma^1_3$-sets $A_m$ and $A_k$ can separate the pair of $\Pi^1_3$-sets $D^0,D^1$. This can be done in countably many steps. After we finished we will work towards a failure of $\Sigma^1_3$-separation, while preserving the fact that $\Pi^1_3$-separation already fails. 

This split of tasks is somewhat necessary as there is some tension between forcing a failure of $\Pi^1_3$ and $\Sigma^1_3$-separation with our technique.

\subsection{Towards a failure of $\Pi^1_3$-separation}
We pick a function $F: \omega_2 \rightarrow H(\omega_2)$ such that $\forall x \in H(\omega_2) F^{-1} (x)$ is unbounded in $\omega_2$. We let $E: H(\omega_2) \rightarrow \{2,3\}$ be arbitrary. The iteration we shall define is guided by both $F$ and $E$. 

Assume that we arrived at stage $\beta<\omega_2$ of the iteration and we obtained already the following
\begin{itemize}
\item The iteration $\forceP_{\beta}$ up to stage $\beta$ and a generic filter $G_{\beta}$ for it.
\item A notion of $\alpha_{\beta}$-suitability, which is a subset of plain suitability; This notion determines a set $N_{\beta}$ of indices of trees from $\vec{S}^i$, $i \in 3$ which we must not use for coding forcings in later factors of our to be defined iteration. This means that the forcing $\operatorname{Code}(x,\eta,i)$ must not be used, for any pair $(\eta,i) \in N_{\beta}$ and any real $x$.

The notion of $\alpha_{\beta}$-suitability also determines a set $R_{\beta} \subset H(\omega_2)^W$ of $\forceP_{\beta}$-names of reals and a function $E_{\beta}: R_{\beta} \rightarrow \{2,3\}$.

\end{itemize}

We assume that the bookkeeping $F$ at $\beta$ hands us $m,k \in \omega$ and the $\forceP_{\beta}$-name of a real $\dot{x}$. We write $x$ for $\dot{x}^{G_{\beta}}$.
Let us suppose that $m,k$ are G\"odelnumbers for two $\Sigma^1_3$-sets $A_m$ and $A_k$ with corresponding $\Sigma^1_3$-formulas $\varphi_m$ and $\varphi_k$.
We also assume that the bookkeeping $F$ visits $x,m,k$ for the very first time, that is there were no decisions already made concering $x,A_m$ and $A_k$ at an earlier stage.
We first ask whether there is an $\alpha_{\beta}$-suitable forcing $\forceQ \in W[G_{\beta}]$ such that 
\[ \forceQ \Vdash x \in A_m \cap A_k. \]
If the answer is yes, then we force with such a forcing $\forceQ= \dot{\forceQ}_{\beta}$ (pick the $<$-least such forcing in some previously fixed well-order). Note that after using $\forceQ$, $A_m$ and $A_k$ will have non-empty intersection for all later stages of our iteration as well, by upwards absoluteness of $\Sigma^1_3$-formulas. Thus, $A_m$ and $A_k$ can not be candidates for separating $D^0,D^1$ and we can move on in our iteration.

If the answer is no, then we argue as follows. In order for $A_m$ and $A_k$ to be potential candidates for separating $D^0,D^1$ we must demand that in the final model $A_m \cup A_k = \omega^{\omega}$. If there is an $\alpha_{\beta}$-suitable $\forceQ= ((\forceQ_{\zeta}, \dot{\mathbb{R}}_{\zeta}) \mid \zeta < \xi<\omega_1)$ guided by some $F'$ and some $E'$ such that $E'$ and $E_{\beta}$ agree on their common domain  and such that
\[ \forceQ \Vdash x \in A_m \]
then we consider the set $T_{\beta} \in W[G_{\beta}]$ of trees from $\vec{S}^i$ which are potentially used for coding by $\forceQ$. More precisely we let
\begin{align*}
T_{\beta} := \{ (\eta,i) \mid \exists  p_{\zeta} \in \forceQ_{\zeta} \exists \dot{x} \in W[G_{\beta}]^{\forceQ_{\zeta}} \\ p_{\zeta} \Vdash (\operatorname{Code} (\dot{x},\eta,i)=\dot{\forceR}_{\zeta} \}
\end{align*}
Note that as all forcings have the ccc $T_{\beta}$ is countable.
We let 
\[N_{\beta+1}:= N_{\beta} \cup T_{\beta}.\]
We also collect all the names of reals which are in the domain of $E'$ and extend $E_{\beta+1}$ in the following way
\begin{align*}
E_{\beta+1}:= E_{\beta} \cup E'
\end{align*}
Note that this results in a partial function $E_{\beta+1}$ from $H(\omega_2)^W \rightarrow \{2,3\}$ again, as we assumed that $E'$ and $E_{\beta}$ agree on their common domain. We define the next notion of suitability.

\begin{definition}
An iteration $\forceP=((\forceP_{\zeta},\dot{\forceQ}_{\zeta}) \mid \zeta < \xi < \omega_2)$ is $\alpha_{\beta+1}$-suitable if $\forceP$ is suitable relative to some $F$ and $E$ and additionally
\begin{itemize}
\item $E$ and $E_{\beta+1}$ agree on their common domain.
\item For every $\zeta< \xi$, every pair of $\forceP_{\zeta}$-names $\dot{x}, \dot{\eta}$ and every $p \in \forceP_{\zeta}$, if $p_{\zeta} \Vdash \dot{\forceQ}_{\zeta} =\operatorname{Code} (\dot{x},\dot{\eta},\dot{i})$, then $p_{\zeta} \Vdash (\dot{\eta},\dot{i}) \notin N_{\beta+1}$.
\end{itemize}  
\end{definition}

The key observation is that we know that no $\alpha_{\beta+1}$-suitable $\forceQ'$ guided by some $F'$ and some $E'$,  will force $x$ to become a member of $A_k$. Indeed otherwise we could form the product $\forceQ \times \forceQ'$ which is $\alpha_{\beta}$-suitable again and, again by upwards absoluteness of $\Sigma^1_3$-formulas, will force $x \in A_m \cap A_k$ which is a contradiction to our prior assumption. 

The very same argument can be applied if we can not force $x$ into $A_m$ but into $A_k$ and along the way we will define the notion of $\alpha_{\beta+1}$-suitable in the same fashion. Hence we obtain the following dichotomy under the assumption that there is no $\alpha_{\beta}$-suitable $\forceQ$ which places $x$ into both $A_m$ and $A_k$ and under the assumption that $A_m$ and $A_k$ are reasonable candidates for separating $D^0,D^1$. Either there is an $\alpha_{\beta+1}$-suitable $\forceQ$, which will force 
$x$ into $A_m$ but no $\alpha_{\beta+1}$-suitable forcing puts $x$ into $A_k$. Or there is an $\alpha_{\beta+1}$-suitable (note that this notion of $\alpha_{\beta+1}$-suitable is not the same as the one from the last sentence) $\forceQ$ which  will force $x$ into $A_k$ but no $\alpha_{\beta+1}$-suitable forcing forces $x$ into $A_m$.

Without loss of generality we assume that $x$ can be forced into $A_k$ with an $\alpha_{\beta+1}$-suitable $\forceQ$ (chosen to be the $<$-least such forcing), but can not be forced into $A_m$. In this situation we will use $\operatorname{Code} (x,\eta, 0)$ for some $\eta$ such that $(\eta,0) \notin N_{\beta+1}$ as the $\beta$-th factor $\dot{\forceQ}_{\beta}$ of our iteration and immediately pass to a fresh real $x' \ne x$ which has not been considered by the bookkeeping yet. 

We repeat the above reasoning with $x'$ instead of $x$ and with the new notion of $\alpha_{\beta}+1$-suitability.
That is, we ask whether $x'$ can be forced into $A_m$ and $A_k$ with an $\alpha_{\beta}+1$-suitable forcing. If yes, then we put $x'$ into $A_m$ and $A_k$. 

If no, then we assume first that there is an $\alpha_{\beta}+1$-suitable forcing $\forceQ$ for which $\forceQ \Vdash x' \in A_m$ and define the notion of $\alpha_{\beta+2}$-suitable in the very same way to our definition of $\alpha_{\beta+1}$-suitability.
Then we force with $\operatorname{Code} (x',\eta',1)$ for some $\eta'$ such that $(\eta',1) \notin N_{\beta+2}$. Note that this implies that $x,x' \in A_m$ yet $x \in D^1$ and $x' \in D^0$, thus $A_m$ and $A_k$ can not separate $D^0$ and $D^1$, as long as we keep $x \in D^1$ and $x' \in D^0$ in our iteration which we can guarantee via defining $\alpha_{\beta}+3$-suitable to be
$\alpha_{\beta}+2$-suitable with the additional demand that we will not use $\operatorname{Code} (x,\eta,0)$ and $\operatorname{Code}(x',\eta,1)$ for all $\eta$ as a factor of our iteration. Finally we let $\alpha_{\beta+1}$-suitable  be just $\alpha_{\beta}+3$-suitable and we continue with $F(\beta+1)$ in our iteration.

If we can not force $x'$ into both $A_m$ and $A_k$ but there is an $\alpha_{\beta}+1$-suitable forcing which forces $x$ into $A_k$, then we argue in the dual way and again obtain first a notion of $\alpha_{\beta+2}$-suitability. Then we decide to force with $\operatorname{Code} (x',\eta',0)$. Note that as a consequence
$x \in A_m \cap D^0$ and $x' \in A_k \cap D^0$, thus $A_m, A_k$ can not separate $D^0,D^1$ as long as we do not use $\operatorname{Code} (x,\eta,0)$ and $\operatorname{Code}(x',\eta,1)$ for all $\eta$ as a factor of our iteration. We define $\alpha_{\beta}+3$-suitable to be
$\alpha_{\beta}+2$-suitable with the additional demand that we will not use $\operatorname{Code} (x,\eta,0)$ and $\operatorname{Code}(x',\eta,1)$ for all $\eta$ as a factor of our iteration. Finally we let $\alpha_{\beta+1}$-suitable  be just $\alpha_{\beta}+3$-suitable and we continue with $F(\beta+1)$ in our iteration.

If $F(\beta)= (x,m,k,i)$ for $m,k$ being G\"odelnumbers of $\Pi^1_3$-sets, we for now, follow what $F$ and $E'=E_{\beta} \upharpoonright \operatorname{dom} (E_{\beta}) \cup E \upharpoonright ( H(\omega_2) \backslash \operatorname{dom} (E_{\beta} )) $ tell us to do.

If $F(\beta)$ is a $\forceP_{\beta}$-name of an $\aleph_1$-sized, ccc indestructible forcing, then we force with it.

Note that, as there are only countably many pairs of $\Sigma^1_3$-sets, after countably many stages $\beta_0$ of this iteration we will arrive at a universe $W[G_{\beta_0}]$ such that 
\[ W[G_{\beta_0}] \models ``D^0,D^1 \text{ can not be separated by a pair of $\Sigma^1_3$-sets.$"$} \]
Moreover the notion $\alpha_{\beta_0}$-suitable is such that in all outer universes of $W[G_{\beta_0}]$ obtained via an $\alpha_{\beta_0}$-suitable forcing, the $\Pi^1_3$-separation property continues to fail. The set of $\alpha_{\beta_0}$-suitable forcings come together with a countable set $N = N_{\beta_0} \in W[G_{\beta_0}]$ of trees from $\vec{S}^i$ which we must not use in all our coding forcings to come; and with a 
partial function $E=E_{\beta_0}$ mapping from $H(\omega_2)^W$ to $\{2,3\}$.

Now we will turn to the $\Sigma^1_3$-separation property.

\subsection{Towards a failure of $\Sigma^1_3$-separation}
We work with $W[G_{\beta_0}]$ as our ground model. We restrict ourselves to only use $\alpha_{\beta_0}$-suitable forcings from now on.
We assume that we arrived at stage $\beta$ of our iteration and we already have
$\forceP_{\beta}$, $G_{\beta} \subset \forceP_{\beta}$, a notion of $\alpha_{\beta}$-suitability together with a function $E_{\beta} : H(\omega_2)^W \rightarrow \{2,3\}$.
Assume that $F(\beta)= (\dot{x},\dot{x'}, m,k,i)$ for $m,k$ being G\"odelnumbers of $\Pi^1_3$-sets $B_m,B_k$, $\dot{x}$, $\dot{x'}$ two $\forceP_{\beta}$-name of two distinct reals $x$ and $x'$, which have not been considered by the bookkeeping before (and moreover are not a factor of one of the virtual forcings $\forceQ$), and $i \in 4$.
We assume that the $<$-least forcing names of $x$ and $x'$ both do not belong to the domain of $E_{\beta}$, which we surely can do.

Next we assume without loss of generality that either 
\begin{enumerate}
\item $x,x' \in B_m$ or 
\item $x \in B_m$ and $x' \in B_k$ or 
\item $x,x' \in B_k$.
\end{enumerate}
This assumption is again harmless, as we must assume that $x,x'$ will end up in $B_m$ or $B_k$ as otherwise $B_m,B_k$ would not partition the reals and hence would not be candidates for separating sets anyway. We will work through the tree cases now.
\begin{enumerate}
\item If $x,x' \in B_m$ we let
\[ \dot{\forceQ}_{\beta} := \operatorname{Code} (x,\eta,2) \times \operatorname{Code} (x',\eta' ,3) \]
for some $\eta, \eta'$ whose trees starting at the $\eta$-th  (and $\eta'$-th)  $\omega$-block of trees from $\vec{S}^i$ have not been used yet for coding by some previous forcing and also  $[\omega \eta, \omega \eta +\omega) \cap N$ and $[\omega \eta', \omega \eta' + \omega) \cap N$ are both empty.

Note that as a consequence $x \in D^2 \cap B_m$ and $x' \in D^3 \cap B_m$, hence $B_m, B_k$ can not separate $D^2,D^3$ as long as $x$ remains in $D^2$ for the rest of the iteration and $x'$ remains in $D^3$ for the rest of the iteration. But this is can easily be accomplished via defining that $E_{\beta_{\alpha}+1} (\tau) = 2$, for $\tau$ the $<$-least forcing name of $x$ and $E_{\beta_{\alpha}+1} (\tau')=3$ for $\tau'$ the $<$-least forcing name of $x'$ and $E_{\beta_{\alpha}+1} =E_{\beta_{\alpha}}$ elsewhere, and using only $\alpha_{\beta}$-suitable forcings with respect to $E$ from now on.

\item If $x \in B_m$ and $x' \in B_k$ we define
\[ \dot{\forceQ}_{\beta} := \operatorname{Code} (x,\eta ,2) \times \operatorname{Code} (x' ,\eta',2) \]
for $\eta$ and $\eta'$ chosen as above. Note that $x,x'$ witness that $B_m,B_k$ can not separate $D^2,D^3$, as long as $x$ stays in $D^2$ and $x'$ stays in $D^3$ for the rest of the iteration. This can easily be accomplished via defining $H'$ in a similar way to above.

\item Last, if $x,x' \in B_k$ then we proceed similar to the first case. We leave the straightforward modifications to the reader.

\end{enumerate}

This ends the definition of what we do when working towards a failure of $\Sigma^1_3$-separation. Note that we will be finished after $\beta_1$-many stages, for some $\beta_1 <\omega_1$, ending up with such a desired universe. Moreover we do have a notion of $\alpha_{\beta_1}$-suitable which ensures that in all further generic extensions, obtained with an $\alpha_{\beta_1}$-suitable forcing, both $\Sigma^1_3$ and $\Pi^1_3$-separation continue to fail.

\subsection{Tail of the iteration}

For the tail of the iteration we will just take care of $\mathsf{MA} (\mathcal{I})$ and ensure that $D^2$ and $D^3$ have non-empty intersection, that is we ensure that for every real $x$, either $x$ is coded into $\vec{S}^2$ or into $\vec{S}^3$.
This can be done with  $\alpha_{\beta_1}$-suitable forcings.

\section{Open Questions}

We end with a couple of natural questions whose answers would need new ideas.

\begin{question}
Is it possible to extend the method of forcing the failure of $\Pi^1_4$-separation to produce universes where $\Sigma$ and $\Pi$-separation fails for all projective levels $\ge 3$?
\end{question}
An interesting question is the relation of mild forcing axioms and the globale failure of projective separation. There is a surprising and substantial tension between both principles, as e.g. $\BPFA$ together with the anti-large cardinal axiom $\omega_1=\omega_1^L$ outright implies that $\Sigma^1_3$-uniformization holds (see \cite{Ho3} for a proof). 
\begin{question}
Is Martin's Axiom consistent with the the  failure of $\Sigma^1_3$ and $\Pi^1_3$-separation? Is Martin;s Axiom consistent with a global failure of projective separation.
\end{question}
Another interesting family of questions which are wide open is to investigate the failure of separation in the presence of regularity properties which are implied by projective determinacy. Note that e.g. in Solovay's model, there is a $\Pi^1_2$-set which can not be uniformized by an ordinal definable function, by L\'evy's argument.
\begin{question}
Given an inacessible, is there a universe where each projective set is Lebesgue measurable and projective separation fails?
\end{question}

\end{document}